\documentclass[11pt]{article}
\usepackage{amsmath}
\usepackage{amssymb}
\usepackage{amsthm}
\usepackage{bookmark}
\usepackage{geometry}
\usepackage{hyperref}
\newtheorem{theorem}{Theorem}
\newtheorem{corollary}[theorem]{Corollary}
\newtheorem{lemma}[theorem]{Lemma}
\newtheorem{proposition}[theorem]{Proposition}
\theoremstyle{definition}
\newtheorem{example}[theorem]{Example}

\DeclareMathOperator{\Gal}{Gal}
\DeclareMathOperator{\lcm}{lcm}
\DeclareMathOperator{\N}{N}
\DeclareMathOperator{\Tr}{Tr}

\title{Characterizations of a Class of Planar Functions over Finite Fields}
\author{Ruikai Chen\textsuperscript{1,2}\and Sihem Mesnager\textsuperscript{1,2,3}}
\date{\small\textsuperscript{1}Department of Mathematics, University of Paris VIII, F-93526 Saint-Denis\\\textsuperscript{2}Laboratory Analysis, Geometry and Applications, LAGA, University Sorbonne Paris Nord, CNRS, UMR 7539, F-93430, Villetaneuse, France\\\textsuperscript{3}Telecom Paris, Polytechnic institute of Paris, 91120 Palaiseau, France\\Emails: \href{mailto:chen.rk@outlook.com}{chen.rk@outlook.com}\quad\href{mailto:smesnager@univ-paris8.fr}{smesnager@univ-paris8.fr}}

\begin{document}

\maketitle

\noindent\textbf{Abstract.} Planar functions, introduced by Dembowski and Ostrom, have attracted much attention in the last decade. As shown in this paper, we present a new class of planar functions of the form $\operatorname{Tr}(ax^{q+1})+\ell(x^2)$ on an extension of the finite field $\mathbb F_{q^n}/\mathbb F_q$. Specifically, we investigate those functions on $\mathbb F_{q^2}/\mathbb F_q$ and construct several typical kinds of planar functions. We also completely characterize them on $\mathbb F_{q^3}/\mathbb F_q$. When the degree of extension is higher, it will be proved that such planar functions do not exist given certain conditions.

\section{Introduction}

Consider an extension $\mathbb F_{q^n}/\mathbb F_q$ of a finite field of odd characteristic. The function induced by a polynomial $f$ over $\mathbb F_{q^n}$ is called a planar function if $f(x+c)-f(x)$ is a permutation on $\mathbb F_{q^n}$ for every $c\in\mathbb F_{q^n}^*$. The notion of planar functions, originally introduced by Dembowski and Ostrom (\cite{dembowski1968planes}), coincides with that of so-called perfect nonlinear functions (see, e.g., \cite{pott2016almost} and \cite{zieve2015planar}) in the case of odd characteristic. Planar functions occur in various areas of mathematics with vast applications in coding theory, cryptography, combinatorics, etc. For instance, planar functions prove helpful in DES-like cryptosystems (\cite{nyberg1992provable}), in error-correcting codes (\cite{carlet2005linear,ding2009explicit,yuan2006weight}), and in signal sets (\cite{ding2007signal}). Besides, planar functions induce many combinatorial objects such as skew Hadamard difference sets, Paley type partial difference sets (\cite{weng2007pseudo}), relative difference sets (\cite{ganley1975relative}) and symplectic spreads (\cite{abdukhalikov2015symplectic}).

Despite the significance of these functions in theory and applications, not many constructions of planar functions have been discovered. Some important planar functions on $\mathbb F_{q^n}$ are listed as follows:
\begin{itemize}
\item $x^{p^k+1}$ with $\frac n{\gcd(k,n)}$ odd (\cite{dembowski1968planes});
\item $x^{10}+x^6-x^2$ with $q=3$ and either $n=2$ or $n$ odd (\cite{coulter1997planar});
\item $x^{\frac{3^k+1}2}$ with $q=3$, $\gcd(k,n)=1$ and $k$ odd (\cite{coulter1997planar});
\item $x^{10}-ux^6-u^2x^2$ with $u\in\mathbb F_{q^n}$, $q=3$ and $n$ odd (\cite{ding2006family}).
\end{itemize}
For other constructions of planar functions and related problems, see \cite{coulter2008commutative,zha2009perfect,bartoli2022planar,bergman2022classifying} and those references therein.

To extend the investigation of planar functions, we introduce a new class of planar functions of the form $\Tr(ax^{q+1})+\ell(x^2)$, where $a\in\mathbb F_{q^n}^*$, $\Tr$ is the trace function from $\mathbb F_{q^n}$ to $\mathbb F_q$, and $\ell$ is an arbitrary linearized polynomial over $\mathbb F_{q^n}$ (a polynomial over $\mathbb F_{q^n}$ that induces a linear endomorphism of $\mathbb F_{q^n}$ over its prime field). In this paper, we study planar functions mentioned above according to the degree $n$ and try our utmost to characterize them. In Section \ref{New-class}, we focus on the case $n=2$, where many instances of planar functions exist, and also the case $n=3$, where all those planar functions are characterized explicitly. Next, Section \ref{Non-existence-Results} shows that as the degree $n$ gets higher, those planar functions tend not to exist.

The following notation will be used throughout. For the finite field $\mathbb F_{q^n}$ with odd characteristic $p$, let $\Tr$ and $\N$ denote the trace and the norm function from $\mathbb F_{q^n}$ to $\mathbb F_q$ respectively. Consider the polynomial ring $\mathbb F_{q^n}[x]$ with an indeterminate $x$. When dealing with functions on $\mathbb F_{q^n}$, only polynomials of degree less than $q^n$ are discussed, all of which are treated modulo $x^{q^n}-x$. 

\section{The New Class of Planar Functions}\label{New-class}

Let $f(x)=\Tr(ax^{q+1})+\ell(x^2)$ for $a\in\mathbb F_{q^n}^*$ with a linearized polynomial $\ell$ over $\mathbb F_{q^n}$. Then $f$ is a planar function on $\mathbb F_{q^n}$ if and only if for all $v\in\mathbb F_{q^n}^*$,
\[f(x+v)-f(x)-f(v)=\Tr(avx^q+av^qx)+2\ell(vx)\]
induces a permutation of $\mathbb F_{q^n}$; i.e., it has no root in $\mathbb F_{q^n}^*$. By substituting $uv^{-1}$ for $x$ with $u\in\mathbb F_{q^n}^*$, we obtain that the function $f$ is planar if and only if
\[\Tr(au^qv^{1-q}+auv^{q-1})+2\ell(u)\ne0\]
for all $u,v\in\mathbb F_{q^n}^*$. As $\Tr(au^qv^{1-q}+auv^{q-1})\in\mathbb F_q$, only those $u\in\mathbb F_{q^n}^*$ with $\ell(u)\in\mathbb F_q$ need to be considered. It will be seen that the problem varies for different degrees $n$.

\subsection{Planar Functions on Quadratic Extensions}\label{planar-quadratic-Extensions}

Let $n=2$, and $f(x)=\Tr(a)x^{q+1}+\ell(x^2)$. If $\Tr(a)=0$, then, by definition, it is planar if and only if $\ell$ is a permutation polynomial of $\mathbb F_{q^2}$. Hence, we only consider $f(x)=x^{q+1}+\ell(x^2)$. The function is planar if and only if
\[\Tr(uv^{q-1})+2\ell(u)\ne0\]
for all $u,v\in\mathbb F_{q^2}^*$. Fix some $u\in\mathbb F_{q^2}^*$ such that $\ell(u)\in\mathbb F_q$. If $\Tr(uv^{q-1})+2\ell(u)=0$ for some $v\in\mathbb F_{q^2}^*$, then we have an element $\alpha=uv^{q-1}$ with $\N(\alpha)=\N(u)$ and $\Tr(\alpha)=-2\ell(u)$, so that the polynomial
\[x^2+2\ell(u)x+\N(u)=x^2-\Tr(\alpha)x+\N(\alpha)\]
has a multiple root $\alpha$ in $\mathbb F_q$, or is irreducible over $\mathbb F_q$. The converse is obvious by the same argument, since an element $\alpha\in\mathbb F_{q^2}$ with $\N(\alpha)=\N(u)$ can be written as $\alpha=uv^{q-1}$ for some $v\in\mathbb F_{q^2}^*$. Hence it follows from the discriminant of the polynomial $x^2+2\ell(u)x+\N(u)$ that $\Tr(uv^{q-1})+2\ell(u)\ne0$ for all $v\in\mathbb F_{q^2}^*$ if and only if $\ell(u)^2-\N(u)$ is a nonzero square in $\mathbb F_q$. This proves the following.

\begin{proposition}\label{1}
The function $x^{q+1}+\ell(x^2)$ is planar on $\mathbb F_{q^2}$ if and only if $\ell(u)^2-\N(u)$ is a nonzero square in $\mathbb F_q$ for all $u\in\mathbb F_{q^2}^*$ such that $\ell(u)\in\mathbb F_q$.
\end{proposition}

For such linearized polynomial $\ell$ over $\mathbb F_{q^2}$, we are interested in those elements $u\in\mathbb F_{q^2}^*$ such that $\ell(u)\in\mathbb F_q$, which form a vector subspace of $\mathbb F_{q^2}$ containing the kernel of $\ell$. Before showing the properties of this subspace, we present a result that will be helpful afterwards. 

\begin{lemma}\label{subspace}
Every $k$-dimensional vector subspace of $\mathbb F_{p^m}$ over $\mathbb F_p$ is equal to the image of some linearized polynomial over $\mathbb F_{p^m}$ of degree $p^{m-k}$.
\end{lemma}
\begin{proof}
Let $V$ be a subspace of dimension $m-k$ and $W$ an associated subspace defined as the image of the linear endomorphism on $\mathbb F_{p^m}$ induced by
\[g(x)=\prod_{\xi\in V}(x-\xi)=x^{p^{m-k}}+\alpha_{m-k-1}x^{p^{m-k-1}}+\dots+\alpha_0x.\]
Fix $W$ and let
\[h(x)=\prod_{\xi\in W}(x-\xi)=x^{p^k}+\beta_{k-1}x^{p^{k-1}}+\dots+\beta_0x.\]
It follows that the composition of $h$ and $g$ is the zero map and has degree $p^m$ as a polynomial. Then $h(g(x))=x^{p^m}-x$ formally. By comparing the coefficients,
\[\alpha_{m-k-1}^{p^k}+\beta_{k-1}=0,\qquad \alpha_{m-k-2}^{p^k}+\beta_{k-1}\alpha_{m-k-1}^{p^{k-1}}+\beta_{k-2}=0,\]
and proceeding by induction we obtain that $g$ is uniquely determined by $h$. Therefore, there is a unique subspace of dimension $m-k$ corresponding to $W$ in the above way. This holds for arbitrary $W$ of dimension $k$, since the number of vector subspaces of dimension $k$ of $\mathbb F_{q^m}$ over $\mathbb F_p$ is equal to that of dimension $m-k$.
\end{proof}

\begin{proposition}\label{kernel}
Suppose the function $x^{q+1}+\ell(x^2)$ on $\mathbb F_{q^2}$ is planar. Then $\ell$ has at most $q$ roots in $\mathbb F_{q^2}$. Moreover, $\ell$ has exactly $q$ roots in $\mathbb F_{q^2}$ if $\ell(\mathbb F_{q^2})\cap\mathbb F_q=\{0\}$.
\end{proposition}
\begin{proof}
Let $W$ be the kernel of $\ell$ in $\mathbb F_{q^2}$ whose dimension over $\mathbb F_p$ is $k$. For each nonzero $u\in W$, we have $\ell(u)\in\mathbb F_q$, and thus $-\N(u)=\ell(u)^2-\N(u)$ is a nonzero square in $\mathbb F_q$. Let $\eta$ be the quadratic character of $\mathbb F_{q^2}$, which is the composition of the quadratic character of $\mathbb F_q$ and the norm function of $\mathbb F_{q^2}/\mathbb F_q$. Then
\[\left|\sum_{u\in W}\eta(u)\right|=p^k-1.\]
By Lemma \ref{subspace}, there exists a linearized polynomial $g$ over $\mathbb F_{q^2}$ of degree $p^{-k}q^2$ whose image is $W$, so that
\[\left|\sum_{\xi\in\mathbb F_{q^2}}\eta(g(\xi))\right|=p^{-k}q^2\left|\sum_{u\in W}\eta(u)\right|=q^2-p^{-k}q^2.\]
By the Weil bound for character sums (\cite[Theorem 5.41]{lidl1997}),
\[q^2-p^{-k}q^2=\left|\sum_{\xi\in\mathbb F_{q^2}}\eta(g(\xi))\right|\le(p^{-k}q^2-1)q,\]
and thus $p^k\le q$.

As for the second statement, since the image of $\ell$ in $\mathbb F_{q^2}$ has $p^{-k}q^2$ elements, the direct sum of $\ell(\mathbb F_{q^2})$ and $\mathbb F_q$ has $p^{-k}q^3$ elements under the condition $\ell(\mathbb F_{q^2})\cap\mathbb F_q=\{0\}$. As a result, we have $p^{-k}q^3\le q^2$ and then $p^k\ge q$.
\end{proof}

Let us look at the simplest case where $\ell$ is a monomial. Assume $\deg(\ell)\le q$ (otherwise consider the equivalent form $x^{q+1}+\ell(x^2)^q$). If $\ell(x)=bx$ or $\ell(x)=bx^q$ for some $b\in\mathbb F_{q^2}^*$, then the function is planar if and only if $u^2-\N(b)^{-1}u^2$ is a nonzero square in $\mathbb F_q$ for all $u\in\mathbb F_q^*$. This depends only on $1-\N(b)^{-1}$. The other cases can be characterized in a more general form as follows.

\begin{theorem}\label{f}
Let $f(x)=x^{q+1}+\ell(x^2)$, where $\ell(x)=(bx^q+cx)^{p^k}$ for $q=p^m$, $b,c\in\mathbb F_{q^2}$ with $\N(b)\ne\N(c)$ and some integer $k$ with $0<k<m$. Then $f$ is a planar function on $\mathbb F_{q^2}$ if and only if $p^k\equiv1\pmod4$, $m=2k$ and $\N(b-c^q)^{\frac{p^k+1}2}=-(\N(b)-\N(c))^{p^k+1}$.
\end{theorem}
\begin{proof}
Let $\eta$ be the quadratic character of $\mathbb F_q$. The inverse of $bx^q+cx$ as a permutation of $\mathbb F_{q^2}$ is given by $\gamma(bx^q-c^qx)$, where $\gamma^{-1}=\N(b)-\N(c)$. Then $u^{p^k}=\ell(\gamma(bu^q-c^qu))\in\mathbb F_q$ only if $u\in\mathbb F_q$, in which case
\[\ell(\gamma(bu^q-c^qu))^2-\N(\gamma(bu^q-c^qu))=u^{2p^k}-\gamma^2\N(bu^q-c^qu)=u^{2p^k}-\gamma^2\N(b-c^q)u^2.\]
It follows from Proposition \ref{1} that the function $f$ is planar if and only if $\eta(u^{2p^k}-\gamma^2\N(b-c^q)u^2)=1$ for all $u\in\mathbb F_q^*$. Denote $\delta=\gamma^2\N(b-c^q)$. Note first that
\[\begin{split}&\mathrel{\phantom{=}}\sum_{u\in\mathbb F_q^*}\eta(u^{2p^k}-\delta u^2)\\&=2\sum_{\xi\in\mathbb F_q^*,\eta(\xi)=1}\eta(\xi^{p^k}-\delta\xi)\\&=\sum_{\xi\in\mathbb F_q^*,\eta(\xi)=1}\eta(\xi^{p^k}-\delta\xi)-\sum_{\xi\in\mathbb F_q^*,\eta(\xi)=-1}\eta(\xi^{p^k}-\delta\xi)\\&\mathrel{\phantom{=}}+\sum_{\xi\in\mathbb F_q^*,\eta(\xi)=1}\eta(\xi^{p^k}-\delta\xi)+\sum_{\xi\in\mathbb F_q^*,\eta(\xi)=-1}\eta(\xi^{p^k}-\delta\xi)\\&=\sum_{\xi\in\mathbb F_q^*}\eta(\xi^{-1})\eta(\xi^{p^k}-\delta\xi)+\sum_{\xi\in\mathbb F_q^*}\eta(\xi^{p^k}-\delta\xi).\end{split}\]
If $x^{p^k}-\delta x$ is a permutation polynomial of $\mathbb F_q$, then
\[\sum_{\xi\in\mathbb F_q^*}\eta(\xi^{p^k}-\delta\xi)=0,\]
and by Weil bound for character sums, we have
\[\sum_{u\in\mathbb F_q^*}\eta(u^{2p^k}-\delta u^2)=\sum_{\xi\in\mathbb F_q^*}\eta(\xi^{-1})\eta(\xi^{p^k}-\delta\xi)\le(p^k-2)\sqrt q.\]
On the other hand,
\[\begin{split}&\mathrel{\phantom{=}}\sum_{u\in\mathbb F_q^*}\eta(u^{2p^k}-\delta u^2)\\&=\sum_{\xi\in\mathbb F_q^*}\eta(\xi^{-1})\eta(\xi^{p^k}-\delta\xi)^{p^{m-k}}\\&=\sum_{\xi\in\mathbb F_q^*}\eta(\xi^{-1})\eta(\xi-(\delta\xi)^{p^{m-k}})\\&\le(p^{m-k}-2)\sqrt q,\end{split}\]
so
\[\sum_{u\in\mathbb F_q^*}\eta(u^{2p^k}-\delta u^2)\le(\min\{p^k,p^{m-k}\}-2)\sqrt q\le q-2\sqrt q<q-1,\]
which means the function $f$ is not planar.

Suppose from now on that $x^{p^k}-\delta x$ is not a permutation polynomial of $\mathbb F_q$. Then the kernel of $x^{p^k}-\delta x$ as a linear endomorphism on $\mathbb F_q$ over $\mathbb F_p$ has dimension $\gcd(k,m)$ (as $\gcd(p^k-1,p^m-1)=p^{\gcd(k,m)}-1$). By Lemma \ref{subspace}, there exists a linearized polynomial $g$ over $\mathbb F_q$ of degree $p^{\gcd(k,m)}$ such that
\[\sum_{\xi\in\mathbb F_q^*}\eta(\xi^{p^k}-\delta\xi)=\sum_{\xi\in\mathbb F_q^*}\eta(g(\xi))\le(p^{\gcd(k,m)}-1)\sqrt q.\]
Assume $m$ is odd. If $|m-2k|>1$, then
\[\sum_{u\in\mathbb F_q^*}\eta(u^{2p^k}-\delta u^2)\le(2\min\{p^k,p^{m-k}\}-1)\sqrt q\le(2p^\frac{m-3}2-1)\sqrt q<q-1.\]
If $p>3$, then
\[\sum_{u\in\mathbb F_q^*}\eta(u^{2p^k}-\delta u^2)\le(2\min\{p^k,p^{m-k}\}-1)\sqrt q\le(2p^\frac{m-1}2-1)\sqrt q<q-1.\]
If $|m-2k|=1$ and $m>3$, then $\gcd(k,m)<\frac{m-1}2$ as is easily seen, so that
\[\begin{split}&\mathrel{\phantom{=}}\sum_{u\in\mathbb F_q^*}\eta(u^{2p^k}-\delta u^2)\\&=\sum_{\xi\in\mathbb F_q^*}\eta(\xi^{-1})\eta(\xi^{p^k}-\delta\xi)+\sum_{\xi\in\mathbb F_q^*}\eta(\xi^{p^k}-\delta\xi)\\&\le(p^\frac{m-1}2+p^{\gcd(k,m)}-3)\sqrt q\\&\le(p^\frac{m-1}2+p^\frac{m-3}2-3)\sqrt q\\&=(p^{-\frac12}+p^{-\frac32})q-3\sqrt q\\&<q-1.\end{split}\]
If $p=m=3$, then
\[\sum_{u\in\mathbb F_q^*}\eta(u^{2p^k}-\delta u^2)<q-1,\]
by direct computation.

Now assume $m$ is even. If $m\ne 2k$, then
\[\sum_{u\in\mathbb F_q^*}\eta(u^{2p^k}-\delta u^2)\le(2\min\{p^k,p^{m-k}\}-1)\sqrt q\le(2p^{\frac m2-1}-1)\sqrt q<q-1.\]
Let $m=2k$ and $\delta=\beta^{p^k-1}$ for some $\beta\in\mathbb F_q^*$, so that any root of $x^{p^k}-\delta x$ in $\mathbb F_q$ is a product of $\beta$ and an element of $\mathbb F_{p^k}$. Then $u^{2p^k}-\delta u^2\ne0$ for all $u\in\mathbb F_q^*$ if and only if $\eta(\beta)=-1$, since all elements in $\mathbb F_{p^k}$ are squares in $\mathbb F_q$. Let $\alpha$ be an element in $\mathbb F_q^*$ such that $\alpha^{p^k-1}=-\beta^{p^k-1}$, as $(-1)^{p^k+1}=1$. Then
\[\begin{split}\alpha^{p^k}(u^{2p^k}-\delta u^2)^{p^k}&=-\alpha\beta^{p^k-1}(u^2-\beta^{1-p^k}u^{2p^k})\\&=\alpha(-\beta^{p^k-1}u^2+u^{2p^k})\\&=\alpha(u^{2p^k}-\delta u^2)\end{split}\]
for all $u\in\mathbb F_q^*$. This means $u^{2p^k}-\delta u^2$ is a product of $\alpha^{-1}$ and an element in $\mathbb F_{p^k}$, so $\eta(u^{2p^k}-\delta u^2)\ne-1$ for all $u\in\mathbb F_q^*$ if and only if $\eta(\alpha)=1$. Finally, the function $f$ is planar if and only if $\eta(\alpha)=1$ and $\eta(\beta)=-1$. It is clear that $\delta=\beta^{p^k-1}$ with $\eta(\beta)=-1$ if and only if $\delta^\frac{p^k+1}2=-1$, and in this case,
\[\alpha^\frac{q-1}2=\alpha^{(p^k-1)\frac{p^k+1}2}=(-\beta^{p^k-1})^\frac{p^k+1}2=(-1)^\frac{p^k-1}2.\]
Then the result follows immediately.
\end{proof}

In general, for the function $x^{q+1}+\ell(x^2)$ on $\mathbb F_{q^2}$ where $\ell$ is a linearized permutation polynomial of $\mathbb F_{q^2}$ with inverse $\ell^{-1}$, it is planar if and only if $u^2-\N(\ell^{-1}(u))$ is a nonzero square in $\mathbb F_q$ for every $u\in\mathbb F_q^*$. For construction, one may start by looking for such $\ell^{-1}$ and then determine its inverse. If, in addition, $\ell$ maps $\mathbb F_q$ to itself, then it suffices to consider $\ell(u)^2-u^2$ for $u\in\mathbb F_q^*$.

\begin{example}
Let $q=p^{2k}$ and $\ell(x)=x^{p^{3k}}-x^{p^{2k}}-x^{p^k}-x$. It permutes $\mathbb F_{q^2}$ as
\[\begin{vmatrix}-1&-1&-1&1\\1&-1&-1&-1\\-1&1&-1&-1\\-1&-1&1&-1\end{vmatrix}=16\ne0,\]
according to the discussion preceding Theorem 7.24 in \cite{lidl1997}. Then
\[\ell(u)^2-u^2=(-2u)^2-u^2=3u^2\]
for all $u\in\mathbb F_q^*$. This leads to a planar function on $\mathbb F_{q^2}$. One can easily generalize it by taking an element $\alpha\in\mathbb F_q$ with $\alpha^2-1$ being a nonzero square in $\mathbb F_q$, and a linearized permutation $\ell$ of $\mathbb F_{q^2}$ such that $\ell(u)=\alpha u$ for $u\in\mathbb F_q$.
\end{example}

Apart from those arising from linearized permutations, there are another class of planar functions. If the polynomial $\ell$ satisfies $\ell(\mathbb F_{q^2})\cap\mathbb F_q=\{0\}$ and $-\N(u)$ is a nonzero square in $\mathbb F_q$ for each of its root $u$ in $\mathbb F_{q^2}$, then $x^{q+1}+\ell(x^2)$ is planar function on $\mathbb F_{q^2}$.% Moreover, if there is a linearized permutation polynomial $\ell$ over $\mathbb F_{q^2}$ mapping $\mathbb F_q$ to itself, then $x^{q+1}+\ell(x^2)$ also belongs to this class.

\begin{theorem}
Let $f(x)=x^{q+1}+\ell(x^2)$, where $\ell(x)=(bx^q+cx)^{p^k}-c_0(bx^q+cx)$ for $q=p^m$, $b,c\in\mathbb F_{q^2}^*$ with $\N(b)=\N(c)$, $c_0\in\mathbb F_{q^2}$ and some integer $k$ with $0<k<m$. Then $f$ is a planar function on $\mathbb F_{q^2}$ if and only if
\begin{itemize}
\item $(b^{-1}c)^\frac{q+1}2=-1$,
\item $\omega^{q-1}\ne b^qc^{-1}$ for every $\omega\in\mathbb F_{q^2}^*$ such that $\omega^{p^k-1}=c_0$, and
\item $\omega^{q-1}\ne b^qc^{-1}$ for every $\omega\in\mathbb F_{q^2}^*$ such that $((b^{-1}c^q)^{p^k}-1)\omega^{p^k-1}=b^{-1}c^qc_0^q-c_0$.
\end{itemize}
In particular, if $d=\gcd(p^k-1,q^2-1)$ divides $q-1$ (i.e., $\gcd(k,2m)$ divides $m$), then the last two conditions become
\begin{itemize}
\item $c_0^{(q-1)/d}\ne(b^qc^{-1})^{(p^k-1)/d}$, and
\item $(b^{-1}c^qc_0^q-c_0)^{(q-1)/d}\ne((b^{-1}c^q)^{p^k}-1)^{(q-1)/d}(b^qc^{-1})^{(p^k-1)/d}$.
\end{itemize}
\end{theorem}
\begin{proof}
As $\N(b)=\N(c)$, there exist elements $\alpha$ and $\beta$ of $\mathbb F_{q^2}^*$ with $\alpha^{q-1}=b^qc^{-1}$ and $\beta^{q-1}=-b^{-1}c$, and
\[\alpha^{-q}(bx^q+cx)^q=\alpha^{-1}b^{-q}c(b^qx+c^qx^q)=\alpha^{-1}(bx^q+cx),\]
which means the image of $bx^q+cx$ in $\mathbb F_{q^2}$ is $\alpha\mathbb F_q$, a one-dimensional subspace over $\mathbb F_q$. Moreover, its kernel is $\beta\mathbb F_q$. Assume $-\N(\beta)$ is a nonzero square in $\mathbb F_q$, for otherwise $f$ is not planar as $\ell(u)=0$ for all $u\in\beta\mathbb F_q$. Note that
\[(-\N(\beta))^\frac{q-1}2=(-1)^\frac{q-1}2\beta^\frac{(q-1)(q+1)}2=(-1)^\frac{q-1}2(-b^{-1}c)^\frac{q+1}2=-(b^{-1}c)^\frac{q+1}2.\]

Suppose there exists some $u_0\in\mathbb F_{q^2}$ such that $\ell(u_0)\ne0$ and $\ell(u_0)\in\mathbb F_q$. Then $\beta^qu_0=\N(\beta)\beta^{-1}u_0\notin\mathbb F_q$ and $\Tr(\beta^qu_0)^2-4\N(\beta u_0)$ (the discriminant of the minimal polynomial of $\beta^qu_0$ over $\mathbb F_q$) is a non-square in $\mathbb F_q$. Furthermore, there exists some $u\in\mathbb F_q$ making
\[\ell(u_0+\beta u)^2-\N(u_0+\beta u)=\ell(u_0)^2-\N(u_0)-\Tr(\beta^qu_0)u-\N(\beta)u^2\]
a non-square in $\mathbb F_q$, while $\ell(u_0+\beta u)\in\mathbb F_q$. In fact, if $\Tr(\beta^qu_0)^2-4\N(\beta u_0)+4\N(\beta)\ell(u_0)^2\ne0$, then the polynomial
\[-\N(\beta)x^2-\Tr(\beta^qu_0)x-\N(u_0)+\ell(u_0)^2\]
is not a square in $\mathbb F_q[x]$; otherwise $\Tr(\beta^qu_0)^2-4\N(\beta u_0)=-4\N(\beta)\ell(u_0)^2$, where the right side is a square in $\mathbb F_q$, a contradiction. Therefore, for necessity we have $\ell(\mathbb F_{q^2})\cap\mathbb F_q=\{0\}$. In this case, the function $f$ is planar if and only if the kernel of $\ell$ is precisely $\beta\mathbb F_q$, by Proposition \ref{kernel}.

If $u_0\in\mathbb F_{q^2}^*\setminus\beta\mathbb F_q$ is a root of $\ell$, then $bu_0^q+cu_0$ is a nonzero root of $x^{p^k}-c_0x$, and vice versa. Therefore, the kernel of $\ell$ is $\beta\mathbb F_q$ if and only if $\omega\notin\alpha\mathbb F_q$ for every $\omega\in\mathbb F_{q^2}^*$ such that $\omega^{p^k-1}=c_0$. Assume that $d=\gcd(p^k-1,q^2-1)$ divides $q-1$ so that $d=\gcd(p^k-1,q-1)$. Then it follows from elementary number theory that
\[\begin{split}&\mathrel{\phantom{=}}\gcd((p^k-1)(q-1)/d,q^2-1)\\&=\gcd(\lcm(p^k-1,q-1),q^2-1)\\&=\lcm(\gcd(p^k-1,q^2-1),\gcd(q-1,q^2-1))\\&=\lcm(\gcd(p^k-1,q-1),q-1)\\&=q-1.\end{split}\]
If $u^{(p^k-1)(q-1)/d}=1$ for $u\in\mathbb F_{q^2}^*$, then
\[u^{q-1}=u^{\gcd((p^k-1)(q-1)/d,q^2-1)}=1.\]
The converse is obvious. Therefore, if there exists $\omega\in\mathbb F_{q^2}^*$ such that $\omega^{p^k-1}=c_0$, then $\omega^{q-1}\ne\alpha^{q-1}$ is equivalent to $c_0^{(q-1)/d}\ne(b^qc^{-1})^{(p^k-1)/d}$. Otherwise $\N(c_0)^{(q-1)/d}=c_0^{(q^2-1)/d}\ne1$, while $\N(b^qc^{-1})^{(p^k-1)/d}=1$.

Next, consider $\ell(\mathbb F_{q^2})\cap\mathbb F_q$. Note that
\[(bu^q+cu)^q=b^qu+c^qu^q=b^{-1}c^q(bu^q+cu)\]
for $u\in\mathbb F_{q^2}$. Then
\[\begin{split}\ell(u)^q-\ell(u)&=(bu^q+cu)^{p^kq}-c_0^q(bu^q+cu)^q-(bu^q+cu)^{p^k}+c_0(bu^q+cu)\\&=((b^{-1}c^q)^{p^k}-1)(bu^q+cu)^{p^k}-(b^{-1}c^qc_0^q-c_0)(bu^q+cu).\end{split}\]
Observe that this is similar with $\ell(u)$ formally. Under the condition that the kernel of $\ell$ is $\beta\mathbb F_q$, it is clear that $\ell(\mathbb F_{q^2})\cap\mathbb F_q=\{0\}$ if and only if $\ell(u)^q-\ell(u)=0$ implies $u\in\beta\mathbb F_q$. In the same manner, we obtain the necessary and sufficient condition for that and complete the proof.
\end{proof}

\begin{example}
To make it more explicit, we have some examples indicating that there exist such planar functions. For simplicity, let $b=1$, $c=-1$ and $q\equiv1\pmod4$, so that $(b^{-1}c)^\frac{q+1}2=-1$.
\begin{enumerate}
\item Let $c_0$ be an arbitrary element in $\mathbb F_{q^2}$ such that neither $c_0$ nor $\frac{\Tr(c_0)}2$ is a $(p^k-1)$-st power in $\mathbb F_{q^2}^*$ (e.g., $c_0=0$).
% \item Suppose $k=2$ and $m$ is odd. Let $c_0=\beta_0^{p^2-1}$. Then $\omega^{q-1}=\beta_0^{q-1}\omega^{q-1}=\beta_0^{q-1}\omega^{p-1}$.
\item Suppose $\gcd(k,2m)$ divides $m$. If $c_0\in\mathbb F_{q^2}\setminus\mathbb F_q$, then $c_0^{q-1}\ne1$ while $(b^qc^{-1})^{p^k-1}=(-1)^{p^k-1}=1$. Therefore, we only need to make sure
\[(b^{-1}c^qc_0^q-c_0)^{(q-1)/d}\ne((b^{-1}c^q)^{p^k}-1)^{(q-1)/d}(b^qc^{-1})^{(p^k-1)/d};\]
that is,
\[\left(\frac{\Tr(c_0)}2\right)^{(q-1)/d}\ne(-1)^{(p^k-1)/d},\]
whether $c_0\in\mathbb F_q$ or not.
\end{enumerate}
\end{example}

When discussing planar functions, one is interested in an equivalence relation that preserves the properties of planar functions. Two functions $f$ and $g$ on $\mathbb F_{q^n}$ are called (extended affine) equivalent, if $f(x)=\ell_0(g(\ell_1(x)))+\ell_2(x)$ for some affine permutations $\ell_0$ and $\ell_1$ with an affine function $\ell_2$ on $\mathbb F_{q^n}$.

It is noteworthy that the planar functions introduced above are not equivalent to monomials in general. Let $q=p^m$ and $e$ be an integer such that $0\le e<2m$ with $d=\frac{2m}{\gcd(e,2m)}$ being odd. Suppose that $f(x)=x^{q+1}+\ell(x^2)$ is equivalent to $x^{p^e+1}$, which implies that $\ell_0(f(x))=\ell_1(x)^{p^e+1}$ for some linear permutations $\ell_0$ and $\ell_1$ on $\mathbb F_{q^2}$, where
\[\ell_0(x)=\sum_{i=0}^{2m-1}\alpha_ix^{p^i}\quad\text{and}\quad\ell_1(x)=\sum_{i=0}^{2m-1}\beta_ix^{p^i}.\]
Then
\begin{equation}\label{ell}\sum_{i=0}^{2m-1}\alpha_ix^{p^iq+p^i}+\sum_{i=0}^{2m-1}\alpha_i(\ell(x^2))^{p^i}=\sum_{i=0}^{2m-1}\sum_{j=0}^{2m-1}\beta_i\beta_j^{p^e}x^{p^i+p^{j+e}}.\end{equation}
Suppose $\beta_0\ne0$ without loss of generality.

Consider first the case where $e\ne0$. For $k=0,\dots,2m-1$, the term of exponent $p^{k+e}+p^k$ does not occur on the left side of \eqref{ell}, so
\begin{equation}\label{beta}\beta_k^{p^e+1}+\beta_{k+e}\beta_{k-e}^{p^e}=0,\end{equation}
where the subscripts are taken modulo $2m$. Then we get a sequence $\{\gamma_k\}_{k\in\mathbb Z}=\{\beta_{ke}\}_{k\in\mathbb Z}$ of period $d$, for $d$ is the additive order of $e$ modulo $2m$. If some element of this sequence is zero, then all of them are zero by \eqref{beta}. Hence
\[\gamma_k=-\gamma_{k-1}^{p^e+1}\gamma_{k-2}^{-p^e}.\]
We prove by induction that for $k\ge0$,
\[\gamma_k=(-1)^{k(k-1)/2}\gamma_1^\frac{p^{ke}-1}{p^e-1}\gamma_0^{1-\frac{p^{ke}-1}{p^e-1}}.\]
For $k\in\{0,1\}$, it is obvious. For $k\ge2$, by the inductive assumption one has
\[\begin{split}\gamma_k&=-\left((-1)^{(k-1)(k-2)/2}\gamma_1^\frac{p^{(k-1)e}-1}{p^e-1}\gamma_0^{-\frac{p^{(k-1)e}-p^e}{p^e-1}}\right)^{p^e+1}\\&\mathrel{\phantom{=}}\left((-1)^{(k-2)(k-3)/2}\gamma_1^\frac{p^{(k-2)e}-1}{p^e-1}\gamma_0^{-\frac{p^{(k-2)e}-p^e}{p^e-1}}\right)^{-p^e}\\&=-\gamma_1^\frac{p^{ke}-p^e+p^{(k-1)e}-1}{p^e-1}\gamma_0^{-\frac{p^{ke}-p^{2e}+p^{(k-1)e}-p^e}{p^e-1}}\\&\mathrel{\phantom{=}}(-1)^{k(k-1)/2-1}\gamma_1^\frac{p^e-p^{(k-1)e}}{p^e-1}\gamma_0^{\frac{p^{(k-1)e}-p^{2e}}{p^e-1}}\\&=(-1)^{k(k-1)/2}\gamma_1^\frac{p^{ke}-1}{p^e-1}\gamma_0^{-\frac{p^{ke}-p^e}{p^e-1}},\end{split}\]
as desired. As a result,
\[\gamma_0=\gamma_d=(-1)^{d(d-1)/2}\gamma_1^\frac{p^{de}-1}{p^e-1}\gamma_0^{-\frac{p^{de}-p^e}{p^e-1}}\]
and
\[\gamma_1=\gamma_{d+1}=(-1)^{(d+1)d/2}\gamma_1^\frac{p^{(d+1)e}-1}{p^e-1}\gamma_0^{-\frac{p^{(d+1)e}-p^e}{p^e-1}}.\]
This implies
\[\gamma_0^\frac{p^{de}-1}{p^e-1}=(-1)^{d(d-1)/2}\gamma_1^\frac{p^{de}-1}{p^e-1}\]
and
\[\gamma_0^{p^e\frac{p^{de}-1}{p^e-1}}=(-1)^{(d+1)d/2}\gamma_1^{p^e\frac{p^{de}-1}{p^e-1}},\]
but $(-1)^{d(d-1)/2}\ne(-1)^{(d+1)d/2}$ because $d$ is odd. The contradiction shows that all planar functions of the form $x^{q+1}+\ell(x^2)$ are not equivalent to $x^{p^e+1}$ when $e\not\equiv0\pmod{2m}$.

Let $e=0$. Observe that there is no term of exponent $p^i+1$ on the left side of \eqref{ell} if $i\ne0$ or $i\ne m$. An inspection of the corresponding coefficients on the right side leads to
\[\sum_{i=0}^{2m-1}\alpha_ix^{p^iq+p^i}+\sum_{i=0}^{2m-1}\alpha_i\ell(x^2)^{p^i}=(\beta_mx^q+\beta_0x)^2.\]
For the first sum, we have
\[\sum_{i=0}^{m-1}(\alpha_i+\alpha_{i+m})x^{p^iq+p^i}=\sum_{i=0}^{2m-1}\alpha_ix^{p^iq+p^i}=2\beta_m\beta_0x^{q+1},\]
so $\alpha_0+\alpha_m=2\beta_0\beta_m$ and $\alpha_k+\alpha_{k+m}=0$. Let $\ell(x)=(bx^q+cx)^{p^k}$ with $m=2k$ and $\N(b)\ne\N(c)$ as in Theorem \ref{f}, so that
\[\sum_{i=0}^{2m-1}\alpha_i\ell(x^2)^{p^i}=\sum_{i=0}^{2m-1}\alpha_i(bx^{2q}+cx^2)^{p^{k+i}}=\sum_{i=0}^{2m-1}(\alpha_{i+m}b^{p^{k+i}q}+\alpha_ic^{p^{k+i}})x^{2p^{k+i}}.\]
Now comparing the coefficients on both sides of
\[\sum_{i=0}^{2m-1}\alpha_i\ell(x^2)^{p^i}=\beta_m^2x^{2q}+\beta_0^2x^2.\]
gives
\[-\alpha^k(b-c^q)=\alpha_{k+m}b+\alpha_kc^q=\beta_m^2,\]
\[\alpha_k(b^q-c)=\alpha_kb^q+\alpha_{k+m}c=\beta_0^2,\]
and
\[\alpha_mb^{p^kq}+\alpha_0c^{p^k}=\alpha_0b^{p^k}+\alpha_mc^{p^kq}=0.\]
It follows from he first two equations that $\beta_m\ne0$ as $\beta_0\ne0$, while the last one implies
\[\alpha_0c^{p^k(q+1)}=-\alpha_m(bc)^{p^kq}=\alpha_0b^{p^k(q+1)},\]
with
\[\alpha_mb^{p^k(q+1)}=-\alpha_0(bc)^{p^k}=\alpha_mc^{p^k(q+1)}.\]
We arrive at a contradiction to the assumption $\N(b)\ne\N(c)$ after noticing that $\alpha_0+\alpha_m=2\beta_0\beta_m\ne0$. Therefore, all those planar functions from Theorem \ref{f} are not equivalent to any monomial.

\subsection{Planar Functions on Cubic Extensions}\label{planar-Cubic-Extensions}

Now proceed to the case $n=3$ with $f$ given by
\[f(x)=\Tr(ax^{q+1})+\ell(x^2)=ax^{q+1}+a^qx^{q^2+q}+a^{q^2}x^{q^2+1}+\ell(x^2),\]
with $a\in\mathbb F_{q^3}^*$. All planar functions of this form will be characterized completely.

\begin{lemma}\label{cubic}
For $A,B\in\mathbb F_{q^3}^*$ and $r\in\mathbb F_q$, the function $\Tr(Ax^{q-1}+Bx^{1-q})+r$ has no root in $\mathbb F_{q^3}^*$ if and only if $r=\frac{\N(A)+\N(B)}{AB}\ne0$.
\end{lemma}
\begin{proof}
We claim that the function $\Tr(Ax^{q-1}+Bx^{1-q})+r$ has a root in $\mathbb F_{q^3}^*$ if and only if
\begin{equation}\label{tr_u}\Tr(u)=r\quad\text{and}\quad\N(u)+\N(A)+\N(B)-\Tr(ABu^{q^2})=0\end{equation}
for some $u\in\mathbb F_{q^3}$. Note that
\[\N(u)+\N(A)+\N(B)-\Tr(ABu^{q^2})=\begin{vmatrix}u&A&B^{q^2}\\B&u^q&A^q\\A^{q^2}&B^q&u^{q^2}\end{vmatrix},\]
and it is zero if and only if the linearized polynomial $B^{q^2}x^{q^2}+Ax^q+ux$ does not permute $\mathbb F_{q^3}$. If $\Tr(Ax_0^{q-1}+Bx_0^{1-q})+r=0$ for some $x_0\in\mathbb F_{q^3}^*$, then let $u=-Ax_0^{q-1}-B^{q^2}x_0^{q^2-1}$, so that
\[\Tr(u)=-\Tr(Ax_0^{q-1}+Bx_0^{1-q})=r\quad\text{and}\quad B^{q^2}x_0^{q^2}+Ax_0^q+ux_0=0;\]
the latter means $B^{q^2}x^{q^2}+Ax^q+ux$ is not a permutation polynomial and thus $\N(u)+\N(A)+\N(B)-\Tr(ABu^{q^2})=0$. Conversely, assuming \eqref{tr_u} for some $u\in\mathbb F_{q^3}$, by taking a nonzero root $x_0$ of $B^{q^2}x^{q^2}+Ax^q+ux$ we have
\[\Tr(Ax_0^{q-1}+Bx_0^{1-q})+r=\Tr(Ax_0^{q-1}+B^{q^2}x_0^{q^2-1})+r=-\Tr(u)+r=0.\]
This proves the claim.

% The equation $\Tr(Ax^{q-1}+Bx^{1-q})+r=0$ is equivalent to $Ax^{q-1}+B^{q^2}x^{q^2-1}+u=0$ with $\Tr(u)=r$ for some $u\in\mathbb F_{q^3}$. Here $Ax^{q-1}+B^{q^2}x^{q^2-1}+u$ has a root in $\mathbb F_{q^3}^*$ if and only if $B^{q^2}x^{q^2}+Ax^q+ux$ does not permute $\mathbb F_{q^3}$; that is,
% \[\N(u)+\N(A)+\N(B)-\Tr(ABu^{q^2})=\begin{vmatrix}u&A&B^{q^2}\\B&u^q&A^q\\A^{q^2}&B^q&u^{q^2}\end{vmatrix}=0.\]
% Therefore, the function $\Tr(Ax^{q-1}+Bx^{1-q})+r$ has a root in $\mathbb F_{q^3}^*$ if and only if
% \begin{equation}\label{tr_u}\Tr(u)=r\quad\text{and}\quad\N(u)+\N(A)+\N(B)-\Tr(ABu^{q^2})=0\end{equation}
% for some $u\in\mathbb F_{q^3}$.

Suppose $AB\in\mathbb F_q$. If $r=\frac{\N(A)+\N(B)}{AB}\ne0$, then $\Tr(u)=r$ implies that $u\ne0$ and
\[\begin{split}&\mathrel{\phantom{=}}\N(u)+\N(A)+\N(B)-\Tr(ABu^{q^2})\\&=\N(u)+\N(A)+\N(B)-ABr\\&=\N(u)\ne0.\end{split}\]
If $r=\frac{\N(A)+\N(B)}{AB}=0$, then $u=0$ implies \eqref{tr_u}. If $r\ne\frac{\N(A)+\N(B)}{AB}$, then by \cite[Theorem 5.3]{moisio2008kloosterman}, there exists $u\in\mathbb F_{q^3}$ with $\Tr(u)=r$ and
\[\N(u)=-\N(A)-\N(B)+ABr=-\N(A)-\N(B)+\Tr(ABu^{q^2}),\]
which means \eqref{tr_u} holds.

Suppose $AB\notin\mathbb F_q$ and let $\beta_0,\beta_1,\beta_2$ be the dual basis of $1,AB,(AB)^2$ in $\mathbb F_{q^3}/\mathbb F_q$. For arbitrary $r\in\mathbb F_q$, if $u^{q^2}=r\beta_0+r_1\beta_1+r_2\beta_2$ for some $r_1,r_2\in\mathbb F_q$, then $\Tr(u)=r$ and
\[\N(u)+\N(A)+\N(B)-\Tr(ABu^{q^2})=\N(r\beta_0+r_1\beta_1+r_2\beta_2)+\N(A)+\N(B)-r_1.\]
It remains to show that the right side is zero for some $r_1,r_2\in\mathbb F_q$. Let $y$ be an element in some extension of the rational function field $\mathbb F_q(x)$ with
\begin{equation}\label{eq1}\N^\prime(r\beta_0+x\beta_1+y\beta_2)+\N(A)+\N(B)-x=0,\end{equation}
where
\[\begin{split}&\mathrel{\phantom{=}}\N^\prime(r\beta_0+x\beta_1+y\beta_2)\\&=(r\beta_0+x\beta_1+y\beta_2)(r\beta_0^q+x\beta_1^q+y\beta_2^q)(r\beta_0^{q^2}+x\beta_1^{q^2}+y\beta_2^{q^2}).\end{split}\]
Assume $y\in\overline{\mathbb F_q}(x)$ and let $v_\infty$ be the valuation at the infinite place of $\overline{\mathbb F_q}(x)$. If $v_\infty(y)>-1$, then $v_\infty(\N^\prime(r\beta_0+x\beta_1+y\beta_2))=-3$. If $v_\infty(y)<-1$, then $v_\infty(\N^\prime(r\beta_0+y\beta_1+x\beta_2))=3v_\infty(y)$. Both cases contradict to the fact $v_\infty(x-\N(A)-\N(B))=-1$. Thus we have $v_\infty(y)=-1$. Note that $y$ is integral over $\overline{\mathbb F_q}[x]$, which is an integrally closed domain. Then $y=\lambda x+\mu$ for some $\lambda,\mu\in\overline{\mathbb F_q}$, and
\[\begin{split}&\mathrel{\phantom{=}}\N^\prime(r\beta_0+x\beta_1+y\beta_2)\\&=((\beta_1+\lambda\beta_2)x+r\beta_0+\mu\beta_2)((\beta_1^q+\lambda\beta_2^q)x+r\beta_0^q+\mu\beta_2^q)((\beta_1^{q^2}+\lambda\beta_2^{q^2})x+r\beta_0^{q^2}+\mu\beta_2^{q^2}).\end{split}\]
Since $v_\infty(\N^\prime(r\beta_0+x\beta_1+y\beta_2))=-1$, without loss of generality we have
\[\beta_1+\lambda\beta_2=\beta_1^q+\lambda\beta_2^q=0.\]
Then $\lambda\in\mathbb F_q$, but $\beta_1$ and $\beta_2$ are linearly independent over $\mathbb F_q$, a contradiction. This shows that $\mathbb F_q(x,y)/\mathbb F_q(x)$ is an extension of function fields of degree $3$ with constant field $\mathbb F_q$.

Consider the corresponding plane projective curve defined by \eqref{eq1}, which is absolutely irreducible, for a cubic polynomial reducible over a field must have a root in it (see also \cite[Corollary 3.6.8]{stichtenoth2009}). As a consequence of \cite{aubry1996weil}, the number $N$ of rational points over $\mathbb F_q$ of the curve satisfies
\[|N-q-1|\le(3-1)(3-2)\sqrt q=2\sqrt q,\]
which means
\[N\ge q+1-2\sqrt q>0,\]
as $q\ge3$. Clearly, there is no point at infinity on this curve, so there exists at least one element in $\mathbb F_{q^3}$ with desired properties.
\end{proof}

\begin{theorem}
Let $\ell$ be defined by
\[\ell(x)=\sum_{i=0}^{m-1}\sum_{j=0}^2b_{ij}x^{p^iq^j}\]
with coefficients in $\mathbb F_{q^3}$ and $q=p^m$. Then $\Tr(ax^{q+1})+\ell(x^2)$ with $a\in\mathbb F_{q^3}^*$ is a planar function on $\mathbb F_{q^3}$ if and only if $\ell$ permutes $\mathbb F_{q^3}$ and
\[\sum_{j=0}^2b_{ij}a^{2p^iq^{j+1}}=\begin{cases}0&\text{ if }i\ne0,\\\N(a)&\text{ if }i=0.\end{cases}\]
\end{theorem}
\begin{proof}
As mentioned before, the function $\Tr(ax^{q+1})+\ell(x^2)$is planar if and only if 
\[\Tr(au^qv^{1-q}+auv^{q-1})+2\ell(u)\ne0\]
for all $u,v\in\mathbb F_{q^3}^*$. That is, by the preceding lemma,
\[2\ell(u)=\frac{\N(au^q)+\N(au)}{a^2u^{q+1}}\ne0\]
whenever $\ell(u)\in\mathbb F_q$ with $u\in\mathbb F_{q^3}^*$. Apparently $\ell$ necessarily has no nonzero root in $\mathbb F_{q^3}$. Also note that
\[\frac{\N(au^q)+\N(au)}{a^2u^{q+1}}=\frac{2\N(au)}{a^2u^{q+1}}=2\N(a)a^{-2}u^{q^2}.\]
Assuming $\ell(u)=\N(a)a^{-2}u^{q^2}$, it is clear that $\ell(u)\in\mathbb F_q$ implies $a^{-2}u^{q^2}\in\mathbb F_q$, and vice versa. Hence $f$ is planar if and only if
\[\ell(a^{2q}w)=\N(a)a^{-2}(a^{2q}w)^{q^2}=\N(a)w\]
for all $w\in\mathbb F_q^*$, while
\[\ell(a^{2q}w)=\sum_{i=0}^{m-1}\sum_{j=0}^2b_{ij}(a^{2q}w)^{p^iq^j}=\sum_{i=0}^{m-1}\sum_{j=0}^2b_{ij}a^{2p^iq^{j+1}}w^{p^i}.\]
If the polynomial
\[\sum_{i=0}^{m-1}\sum_{j=0}^2b_{ij}a^{2p^iq^{j+1}}x^{p^i}-\N(a)x\]
of degree less than $q$ over $\mathbb F_{q^3}$ has $q$ distinct roots, then it has to be zero. This completes the proof.
\end{proof}

\section{Non-existence Results of Higher Degree}\label{Non-existence-Results}

As shown before, the function defined by
\[f(x)=\Tr(ax^{q+1})+\ell(x^2)\]
over $\mathbb F_{q^n}$ is planar if and only if
\[\Tr(au^qx^{1-q}+aux^{q-1})+2\ell(u)\]
has no root in $\mathbb F_{q^n}^*$ for all $u\in\mathbb F_{q^n}^*$. It is natural to ask whether those planar functions on $\mathbb F_{q^n}$ exist for larger degree $n$. To this end, one may be interested in a generalization of Lemma \ref{cubic}, on which we will concentrate in the subsequent part.

Let us first recall the definition of Artin $L$-series for an abelian extension $E/F$ of global function fields. Let $\mathfrak p$ be a prime of $F$ and $\mathfrak P$ a prime of $E$ lying above $\mathfrak p$, with valuation rings $\mathcal O_\mathfrak p$ and $\mathcal O_\mathfrak P$ respectively. Let $I_\mathfrak p$ be the inertia group of $\mathfrak p$ in $E/F$. There exists a unique coset of $I_\mathfrak p$ in $\Gal(E/F)$ with a representative $\varphi_\mathfrak p$ (called a Frobenius automorphism) such that $\varphi_\mathfrak p(y)\equiv y^{|\mathcal O_\mathfrak p/\mathfrak p|}\pmod{\mathfrak P}$ for every $y\in\mathcal O_\mathfrak P$. For a character $\rho$ of $\Gal(E/F)$, let
\[\rho^\prime(\mathfrak p)=\begin{cases}\rho(\varphi_\mathfrak p)&\text{if }I_\mathfrak p\subseteq\ker\rho,\\0&\text{otherwise,}\end{cases}\]
where $\rho(\varphi_\mathfrak p)$ does not depend on the choice of $\varphi_\mathfrak p$ in that case. The Artin $L$-series of $E/F$ for $\rho$ is defined as
\[\prod_{\mathfrak p}\frac1{1-\rho^\prime(\mathfrak p)z^{\deg(\mathfrak p)}}\]
for $z\in\mathbb C$, where the product is taken over all primes $\mathfrak p$ of $F$.

\begin{lemma}
Let $\chi$ be an additive character of $\mathbb F_q$, and $\psi$ a multiplicative character of $\mathbb F_q$. For a monic polynomial $g$ over $\mathbb F_q$ of degree $m$ given by
\[g(x)=x^m-\alpha_{m-1}x^{m-1}+\dots+(-1)^{m-1}\alpha_1x+(-1)^m\alpha_0,\]
let
\[\Phi_{\chi,\psi}(g)=\begin{cases}\chi(\alpha_{m-1}+c\alpha_1\alpha_0^{-1})\psi(\alpha_0)&\text{if }\alpha_0\ne0,\\0&\text{if }\alpha_0=0,\end{cases}\]
with $c\in\mathbb F_q^*$ fixed. Define
\[L(z,\chi,\psi)=\prod_P\frac1{1-\Phi_{\chi,\psi}(P)z^{\deg(P)}}\]
for $z\in\mathbb C$, where the product is taken over all monic irreducible polynomials $P$ over $\mathbb F_q$. If $\chi$ or $\psi$ is nontrivial, then $L(z,\chi,\psi)$ converges absolutely for $|z|<q^{-1}$; moreover, it is a polynomial in $z$ with complex coefficients, each of whose roots has absolute value $q^{-\frac12}$.
\end{lemma}
\begin{proof}
We show that $L(z,\chi,\psi)$ is equal to some Artin $L$-series. Consider the rational function field $F=\mathbb F_q(x)$ and its extensions $E_0=\mathbb F_q(x,y_0)$ and $E_1=\mathbb F_q(x,y_1)$ defined by
\begin{equation}\label{y}y_0^q-y_0=x+cx^{-1}\quad\text{and}\quad y_1^{q-1}=x,\end{equation}
where $E_0/F$ is an Artin-Schreier extension and $E_1/F$ is a Kummer extension (see \cite[Section 3.7]{stichtenoth2009}). Clearly, the prime generated by $x$ and the infinite prime of $F$ are totally ramified in both $E_0$ and $E_1$, and all other primes of $F$ are unramified. This is also true for $E=E_0E_1$ by \cite[Theorem 3.9.1]{stichtenoth2009}. The constant field of $E/F$ can be embedded into the residue field of any prime of $E$, while there is a prime of $F$ of degree one totally ramified in $E$. Looking at the degree of the residue field extension, we get that the constant field of $E/F$ is $\mathbb F_q$.

Note that $E_0/F$ and $E_1/F$ are abelian extensions and $E_0\cap E_1=F$, so $E$ is also an abelian extension over $F$ with Galois group isomorphic to $\mathbb F_q^+\times\mathbb F_q^*$. Let $\varphi$ be an automorphism in $\Gal(E/F)$ determined by
\[\varphi|_{E_0}:y_0\mapsto y_0+\upsilon\quad\text{and}\quad\varphi|_{E_1}:y_1\mapsto\omega y_1\]
for $\upsilon\in\mathbb F_q$ and $\omega\in\mathbb F_q^*$. By abuse of notation $(\chi,\psi)$ acts as a character of $\Gal(E/F)$ with $(\chi,\psi)(\varphi)=\chi(\upsilon)\psi(\omega)$. For a prime $\mathfrak p$ of $F$, choose a prime $\mathfrak P$ of $E$ lying above $\mathfrak p$. If $\mathfrak p$ is ramified in $E/F$, then $I_\mathfrak p=\Gal(E/F)$, and by the assumption that $\chi$ or $\psi$ is nontrivial, the corresponding factor is trivial in the Artin $L$-series of $E/F$ for $(\chi,\psi)$. Suppose that $\mathfrak p$ is unramified in $E/F$ corresponding to an irreducible polynomial $P$ over $\mathbb F_q$, and $\varphi$ is the Frobenius automorphism of $\mathfrak p$. Then $y_0,y_1\in\mathcal O_\mathfrak P$, since by \eqref{y} they are integral over $\mathcal O_\mathfrak P$, an integrally closed domain. Then
\[\begin{split}\sum_{i=0}^{\deg(\mathfrak p)-1}(x+cx^{-1})^{q^i}&=\sum_{i=0}^{\deg(\mathfrak p)-1}(y_0^q-y_0)^{q^i}=\sum_{i=0}^{\deg(\mathfrak p)-1}(y_0^{q^{i+1}}-y_0^{q^i})=y_0^{q^{\deg(\mathfrak p)}}-y_0\\&\equiv\varphi(y_0)-y_0\equiv\upsilon\pmod{\mathfrak P},\end{split}\]
and
\[\prod_{i=0}^{\deg(\mathfrak p)-1}x^{q^i}=\prod_{i=0}^{\deg(\mathfrak p)-1}y_1^{q^{i+1}-q^i}=y_1^{q^{\deg(\mathfrak p)}-1}\equiv\varphi(y_1)y_1^{-1}\equiv\omega\pmod{\mathfrak P}.\]
Furthermore, since $\mathfrak P\cap F=\mathfrak p$ we have
\[\sum_{i=0}^{\deg(\mathfrak p)-1}(x+cx^{-1})^{q^i}\equiv\upsilon\pmod{\mathfrak p},\]
and
\[\prod_{i=0}^{\deg(\mathfrak p)-1}x^{q^i}\equiv\omega\pmod{\mathfrak p}.\]
The fact that $P$ is the minimal polynomial of $x(\mathfrak p)$ over $\mathbb F_q$ gives rise to
\[\Phi_{\chi,\psi}(P)=\chi(\upsilon)\psi(\omega)=(\chi,\psi)(\varphi),\]
which means $L(z,\chi,\psi)$ coincides with the Artin $L$-series. The result follows from \cite[Proposition 9.15, Theorem 9.16B]{rosen2002}.
\end{proof}

Having obtained information about the roots of the $L$-series, we are able to utilize a classical method from analytic number theory to prove an auxiliary result.

\begin{proposition}
Given $c\in\mathbb F_q^*$ and $n>4$, there exists $\xi\in\mathbb F_{q^n}^*$ with
\[\Tr(\xi+c\xi^{-1})+\upsilon=0\quad\text{and}\quad\omega\N(\xi)=1\]
for any $\upsilon\in\mathbb F_q$ and $\omega\in\mathbb F_q^*$.
\end{proposition}
\begin{proof}
Continue with the notation in the previous lemma. Taking logarithmic derivative of $L(z,\chi,\psi)$ yields
\[\begin{split}&\mathrel{\phantom{=}}z\frac{\mathrm d\log L(z,\chi,\psi)}{\mathrm dz}\\&=\sum_P\frac{\deg(P)\Phi_{\chi,\psi}(P)z^{\deg(P)}}{1-\Phi_{\chi,\psi}(P)z^{\deg(P)}}\\&=\sum_P\deg(P)\Phi_{\chi,\psi}(P)z^{\deg(P)}\sum_{j=0}^\infty(\Phi_{\chi,\psi}(P)z^{\deg(P)})^j\\&=\sum_P\deg(P)\sum_{j=1}^\infty\Phi_{\chi,\psi}(P)^jz^{\deg(P)j}.\end{split}\]
Here, by collecting all the coefficients of $z^k$ one gets
\[z\frac{\mathrm d\log L(z,\chi,\psi)}{\mathrm dz}=\sum_{k=1}^\infty A_{\chi,\psi}(k)z^k,\]
where
\[A_{\chi,\psi}(k)=\sum_P\deg(P)\Phi_{\chi,\psi}(P)^{k/\deg(P)},\]
and the sum is over all monic irreducible polynomials $P$ over $\mathbb F_q$ with degree dividing $k$. If $P$ is the minimal polynomial of $\xi\in\mathbb F_{q^k}$ over $\mathbb F_q$, then
\[P(x)^{k/\deg(P)}=x^k-\Tr_k(\xi)x^{k-1}+\dots+(-1)^{k-1}\Tr_k(\xi^{-1})\N_k(\xi)x+(-1)^k\N_k(\xi),\]
where $\Tr_k$ and $\N_k$ denote the trace and the norm respectively of $\mathbb F_{q^k}/\mathbb F_q$, and thus
\[\Phi_{\chi,\psi}(P)^{k/\deg(P)}=\chi(\Tr_k(\xi)+c\Tr_k(\xi^{-1}))\psi(\N_k(\xi))=\chi(\Tr_k(\xi+c\xi^{-1}))\psi(\N_k(\xi)).\]
The same holds for any conjugate of $\xi$ over $\mathbb F_q$, so
\[A_{\chi,\psi}(k)=\sum_{\xi\in\mathbb F_{q^k}^*}\chi(\Tr_k(\xi+c\xi^{-1}))\psi(\N_k(\xi)).\]
For $\upsilon\in\mathbb F_q$ and $\omega\in\mathbb F_q^*$,
\[\sum_{\chi,\psi}\chi(\upsilon)\psi(\omega)A_{\chi,\psi}(k)=\sum_{\xi\in\mathbb F_{q^k}^*}\sum_{\chi,\psi}\chi(\Tr_k(\xi+c\xi^{-1})+\upsilon)\psi(\N_k(\xi)\omega)=q(q-1)M_k(\upsilon,\omega),\]
where the sum is taken over all additive characters $\chi$ and multiplicative characters $\psi$ of $\mathbb F_q$, and $M_k(\upsilon,\omega)$ is the number of $\xi\in\mathbb F_{q^k}^*$ with $\Tr_k(\xi+c\xi^{-1})+\upsilon=0$ and $\N_k(\xi)\omega=1$.

It remains to evaluate $\sum_{\chi,\psi}\chi(\upsilon)\psi(\omega)z\frac{\mathrm d\log L(z,\chi,\psi)}{\mathrm dz}$ in another way and then arrive at $M_k(\upsilon,\omega)$. Clearly $\Phi_{\chi,\psi}(gh)=\Phi_{\chi,\psi}(g)\Phi_{\chi,\psi}(h)$ for all monic polynomials $g$ and $h$ over $\mathbb F_q$, and $|\Phi_{\chi,\psi}(g)|\le1$. On account of the Euler product formula, we have
\[L(z,\chi,\psi)=\sum_g\Phi_{\chi,\psi}(g)z^{\deg(g)},\]
where the sum is over all monic polynomials $g$ over $\mathbb F_q$. It is convergent for $|z|<q^{-1}$, since
\[L(z,\chi,\psi)=\sum_{k=0}^\infty\left(\sum_{\deg(g)=k}\Phi_{\chi,\psi}(g)\right)z^k,\]
where the coefficient of $z^k$ has absolute value at most $q^k$. Denote by $\chi_0$ and $\psi_0$ the trivial additive and multiplicative characters of $\mathbb F_q$, respectively. Then
\[L(z,\chi_0,\psi_0)=\sum_{g(0)\ne0}z^{\deg(g)}=\sum_gz^{\deg(g)}-\sum_gz^{\deg(g)+1}=(1-z)\sum_{k=0}^\infty q^kz^k=\frac{1-z}{1-qz},\]
and
\[z\frac{\mathrm d\log L(z,\chi_0,\psi_0)}{\mathrm dz}=\frac{qz}{1-qz}-\frac z{1-z}=\sum_{k=1}^\infty(q^k-1)z^k.\]
For $\psi\ne\psi_0$, we have
\[\sum_{\deg(g)=k}\Phi_{\chi_0,\psi}(g)=\sum_{\alpha_1,\dots,\alpha_{k-1}\in\mathbb F_q}\sum_{\alpha_0\in\mathbb F_q^*}\psi(\alpha_0)=0\]
when $k>0$, so
\[L(z,\chi_0,\psi)=1.\]
For $\chi\ne\chi_0$,
\[\begin{split}&\mathrel{\phantom{=}}\sum_{\deg(g)=k}\Phi_{\chi,\psi_0}(g)\\&=\sum_{\alpha_1,\dots,\alpha_{k-1}\in\mathbb F_q}\sum_{\alpha_0\in\mathbb F_q^*}\chi(\alpha_{k-1}+c\alpha_1\alpha_0^{-1})\\&=\sum_{\alpha_2,\dots,\alpha_{k-1}\in\mathbb F_q}\chi(\alpha_{k-1})\sum_{\alpha_1\in\mathbb F_q}\sum_{\alpha_0\in\mathbb F_q^*}\chi(c\alpha_1\alpha_0^{-1})=0\end{split}\]
when $k>2$, and
\[\sum_{\deg(g)=2}\Phi_{\chi,\psi_0}(g)=\sum_{\alpha_1\in\mathbb F_q}\sum_{\alpha_0\in\mathbb F_q^*}\chi(\alpha_1+c\alpha_1\alpha_0^{-1})=\sum_{\alpha_0\in\mathbb F_q^*}\sum_{\alpha_1\in\mathbb F_q}\chi(\alpha_1(1+c\alpha_0^{-1}))=q,\]
for the inner sum in the last equation is zero when $1+c\alpha_0^{-1}\ne0$. Hence $L(z,\chi,\psi_0)$ is a quadratic polynomial in $z$ with constant term $1$, and we may write 
\[L(z,\chi,\psi_0)=(1-\lambda_0(\chi)z)(1-\lambda_1(\chi)z),\]
for complex numbers $\lambda_0(\chi)$ and $\lambda_1(\chi)$ with $|\lambda_0(\chi)|=|\lambda_1(\chi)|=\sqrt q$. Consequently,
\[z\frac{\mathrm d\log L(z,\chi,\psi_0)}{\mathrm dz}=-\frac{\lambda_0(\chi)z}{1-\lambda_0(\chi)z}-\frac{\lambda_1(\chi)z}{1-\lambda_1(\chi)z}=\sum_{k=1}^\infty(-\lambda_0(\chi)^k-\lambda_1(\chi)^k)z^k.\]
For $\chi\ne\chi_0$ and $\psi\ne\psi_0$, likewise we have
\[\sum_{\deg(g)=k}\Phi_{\chi,\psi_0}(g)=\sum_{\alpha_2,\dots,\alpha_{k-1}\in\mathbb F_q}\chi(\alpha_{k-1})\sum_{\alpha_1\in\mathbb F_q}\sum_{\alpha_0\in\mathbb F_q^*}\chi(c\alpha_1\alpha_0^{-1})\psi(\alpha_0)=0\]
when $k>2$, and
\[\sum_{\deg(g)=2}\Phi_{\chi,\psi_0}(g)=\sum_{\alpha_0\in\mathbb F_q^*}\psi(\alpha_0)\sum_{\alpha_1\in\mathbb F_q}\chi(\alpha_1(1+c\alpha_0^{-1}))=\psi(-c)q.\]
It follows that
\[L(z,\chi,\psi)=(1-\mu_0(\chi,\psi)z)(1-\mu_1(\chi,\psi)z)\]
for complex numbers $\mu_0(\chi,\psi)$ and $\mu_1(\chi,\psi)$ with $|\mu_0(\chi,\psi)|=|\mu_1(\chi,\psi)|=\sqrt q$, and then
\[z\frac{\mathrm d\log L(z,\chi,\psi)}{\mathrm dz}=-\frac{\mu_0(\chi,\psi)z}{1-\mu_0(\chi,\psi)z}-\frac{\mu_1(\chi,\psi)z}{1-\mu_1(\chi,\psi)z}=\sum_{k=1}^\infty(-\mu_0(\chi,\psi)^k-\mu_1(\chi,\psi)^k)z^k.\]

Altogether we obtain
\[\begin{split}\sum_{\chi,\psi}\chi(\upsilon)\psi(\omega)A_{\chi,\psi}(k)&=q^k-1-\sum_{\chi\ne\chi_0}\chi(\upsilon)(\lambda_0(\chi)^k+\lambda_1(\chi)^k)\\&\mathrel{\phantom{=}}-\sum_{\chi\ne\chi_0,\psi\ne\psi_0}\chi(\upsilon)\psi(\omega)(\mu_0(\chi,\psi)^k+\mu_1(\chi,\psi)^k),\end{split}\]
and
\[q(q-1)M_k(\upsilon,\omega)\ge q^k-1-2(q-1)q^\frac k2-2(q-1)(q-2)q^\frac k2.\]
It is routine to verify that the right side is positive whenever $q\ge3$ and $k\ge5$.
\end{proof}

\begin{corollary}
Let $f(x)=\Tr(ax^{q+1})+\ell(x^2)$ with a linearized polynomials over $\mathbb F_{q^n}$ with $n>4$. If $a=\alpha^\frac{q+1}2$ for some $\alpha\in\mathbb F_{q^n}^*$ or $n$ is odd, then $f$ is not a planar function on $\mathbb F_{q^n}$.
\end{corollary}
\begin{proof}
If there exists some $u\in\mathbb F_{q^n}^*$ such that $\Tr(au^qx^{1-q}+aux^{q-1})$ takes all values in $\mathbb F_q$, then clearly $f$ is not a planar function on $\mathbb F_{q^n}$. Since
\[\Tr(au^qx^{1-q}+aux^{q-1})=\Tr\left(au^qx^{1-q}+\frac{a^2u^{q+1}}{au^qx^{1-q}}\right),\]
by the above proposition it suffices to show that $a^2u^{q+1}\in\mathbb F_q$ for some $u\in\mathbb F_{q^n}^*$. If $a=\alpha^\frac{q+1}2$ for some $\alpha\in\mathbb F_{q^n}^*$, then
\[(a^2u^{q+1})^{q-1}=a^{2(q-1)}u^{q^2-1}=(\alpha u)^{q^2-1}=1,\]
for $u=\alpha^{-1}$. If $n$ is odd, then $\gcd(q^2-1,q^n-1)=q-1$, and thus $a^{2(q-1)}=u^{1-q^2}$ for some $u\in\mathbb F_{q^n}^*$ with
\[(a^2u^{q+1})^{q-1}=a^{2(q-1)}u^{q^2-1}=1.\]
\end{proof}

As a concluding remark, we have discussed planar functions $\Tr(ax^{q+1})+\ell(x^2)$ for $n=2$ or $n=3$, as well as non-existence results for $n\ge5$. However, the case $n=4$ remains unclear. For $n\ge4$, given arbitrary $A,B\in\mathbb F_{q^n}^*$, it is likely that $\Tr(Ax^{q-1}+Bx^{1-q})$ maps $\mathbb F_{q^n}^*$ onto $\mathbb F_q$. If this occurs, then those planar functions do not exist whenever $n\ge4$. Besides, it is also interesting to find more planar functions in the case $n=2$ using Proposition \ref{1}.

\section*{Acknowledgement}

The authors sincerely thank the anonymous referees for their valuable comments. The work of the first author was supported by the China Scholarship Council. The funding corresponds to the scholarship for the PhD thesis of the first author in Paris, France. The French Agence Nationale de la Recherche partially supported the second author's work through ANR BARRACUDA (ANR-21-CE39-0009).

\bibliographystyle{abbrv}
\bibliography{references}

\end{document}